%% file: rationalcurves_final.tex
\documentclass[11pt]{amsart}

\usepackage{graphpap, color}

\usepackage[mathscr]{eucal}
\usepackage{mathrsfs}

\usepackage{color}
\usepackage{cancel}

\usepackage{dsfont}
\usepackage{amsmath,amsthm,amssymb,amsxtra,amsfonts,amsbsy,mathrsfs,url}
\usepackage{verbatim}
\usepackage[all]{xy}

\usepackage{graphics}

\begin{document}
%
%
%
\theoremstyle{definition}
\newtheorem{Definition}{Definition}[section]
\newtheorem*{Definitionx}{Definition}
\newtheorem{Convention}{Definition}[section]
\newtheorem{Construction}{Construction}[section]
\newtheorem{Example}[Definition]{Example}
\newtheorem{Examples}[Definition]{Examples}
\newtheorem{Remark}[Definition]{Remark}
\newtheorem{Remarks}[Definition]{Remarks}
\newtheorem{Assumption}[Definition]{Assumption}
\newtheorem{Caution}[Definition]{Caution}
\newtheorem{Conjecture}[Definition]{Conjecture}
\newtheorem*{Conjecturex}{Conjecture}
\newtheorem{Question}[Definition]{Question}
\newtheorem{Questions}[Definition]{Questions}
\newtheorem*{Acknowledgements}{Acknowledgements}
\theoremstyle{plain}
\newtheorem{Theorem}[Definition]{Theorem}
\newtheorem*{Theoremx}{Theorem}
\newtheorem{Proposition}[Definition]{Proposition}
\newtheorem*{Propositionx}{Proposition}
\newtheorem{Lemma}[Definition]{Lemma}
\newtheorem{Corollary}[Definition]{Corollary}
\newtheorem*{Corollaryx}{Corollary}
\newtheorem{Fact}[Definition]{Fact}
\newtheorem{Facts}[Definition]{Facts}
\newtheoremstyle{voiditstyle}{3pt}{3pt}{\itshape}{\parindent}%
{\bfseries}{.}{ }{\thmnote{#3}}%
\theoremstyle{voiditstyle}
\newtheorem*{VoidItalic}{}
\newtheoremstyle{voidromstyle}{3pt}{3pt}{\rm}{\parindent}%
{\bfseries}{.}{ }{\thmnote{#3}}%
\theoremstyle{voidromstyle}
\newtheorem*{VoidRoman}{}

%
\newcommand{\prf}{\par\noindent{\sc Proof.}\quad}
\newcommand{\blowup}{\rule[-3mm]{0mm}{0mm}}
\newcommand{\cal}{\mathcal}
\newcommand{\Aff}{{\mathds{A}}}
\newcommand{\BB}{{\mathds{B}}}
\newcommand{\CC}{{\mathds{C}}}
\newcommand{\FF}{{\mathds{F}}}
\newcommand{\GG}{{\mathds{G}}}
\newcommand{\HH}{{\mathds{H}}}
\newcommand{\NN}{{\mathds{N}}}
\newcommand{\ZZ}{{\mathds{Z}}}
\newcommand{\PP}{{\mathds{P}}}
\newcommand{\QQ}{{\mathds{Q}}}
\newcommand{\oQQ}{{\overline{\QQ}}}
\newcommand{\RR}{{\mathds{R}}}
\newcommand{\Liea}{{\mathfrak a}}
\newcommand{\Lieb}{{\mathfrak b}}
\newcommand{\Lieg}{{\mathfrak g}}
\newcommand{\Liem}{{\mathfrak m}}
\newcommand{\ideala}{{\mathfrak a}}
\newcommand{\idealb}{{\mathfrak b}}
\newcommand{\idealg}{{\mathfrak g}}
\newcommand{\idealm}{{\mathfrak m}}
\newcommand{\idealp}{{\mathfrak p}}
\newcommand{\idealq}{{\mathfrak q}}
\newcommand{\idealI}{{\cal I}}
\newcommand{\lin}{\sim}
\newcommand{\num}{\equiv}
\newcommand{\dual}{\ast}
\newcommand{\iso}{\cong}
\newcommand{\homeo}{\approx}
\newcommand{\mm}{{\mathfrak m}}
\newcommand{\pp}{{\mathfrak p}}
\newcommand{\qq}{{\mathfrak q}}
\newcommand{\rr}{{\mathfrak r}}
\newcommand{\pP}{{\mathfrak P}}
\newcommand{\qQ}{{\mathfrak Q}}
\newcommand{\rR}{{\mathfrak R}}
%
%
\newcommand{\dq}{{``}}
\newcommand{\OO}{{\cal O}}
\newcommand{\into}{{\hookrightarrow}}
\newcommand{\onto}{{\twoheadrightarrow}}
\newcommand{\Spec}{{\rm Spec}\:}
\newcommand{\BigSpec}{{\rm\bf Spec}\:}
\newcommand{\Spf}{{\rm Spf}\:}
\newcommand{\Proj}{{\rm Proj}\:}
\newcommand{\Pic}{{\rm Pic }}
\newcommand{\Br}{{\rm Br}}
\newcommand{\NS}{{\rm NS}}
\newcommand{\chit}{\chi_{\rm top}}
\newcommand{\KanDiv}{{\cal K}}
\newcommand{\Cycl}[1]{{\ZZ/{#1}\ZZ}}
\newcommand{\Sym}{{\mathfrak S}}
\newcommand{\ab}{{\rm ab}}
\newcommand{\Aut}{{\rm Aut}}
\newcommand{\Hom}{{\rm Hom}}
\newcommand{\ord}{{\rm ord}}
\newcommand{\Alb}{{\rm Alb}}
\newcommand{\Jac}{{\rm Jac}}
\newcommand{\defin}[1]{{\bf #1}}
\newcommand{\NE}{{\rm NE}}
\newcommand{\oNE}{\overline{{\rm NE}}}

\def\cC{\mathcal C}
\def\Po{\mathbb P^1}
\let\sub=\subset
\makeatletter
\newcommand{\subjclassname@NEW}{2010 Mathematics Subject Classification}
\makeatother

\title[Rational Curves]{Rational Curves on K3 Surfaces}

\author[Jun Li]{Jun Li}
\address{Department of Mathematics, Stanford University, Stanford CA 94305-2125, USA}
\curraddr{}
\email{jli@math.stanford.edu}

\author[Christian Liedtke]{Christian Liedtke}
\address{Department of Mathematics, Stanford University, Stanford CA 94305-2125, USA}
\curraddr{}
\email{liedtke@math.stanford.edu}

\date{\today}
\subjclass[NEW]{14J28, 14N35, 14G17}
\maketitle

\begin{abstract}
 We show that projective K3 surfaces with odd Picard rank 
 contain infinitely many rational curves.
 Our proof extends the Bogomolov-Hassett-Tschinkel approach, 
 i.e., uses moduli spaces of stable maps and 
 reduction to positive characteristic.
\end{abstract}

\section*{Introduction}

For a complex variety of general type, Lang's conjecture \cite{lang}
predicts that all rational curves are contained in a proper algebraic set.
On the other extreme, varieties of negative Kodaira dimension are
conjecturally uniruled, and these contain moving 
families of rational curves through every general point.
In between these two extremes lie varieties with trivial canonical
sheaves, e.g., K3 surfaces, Calabi-Yau manifolds and Abelian varieties.
Abelian varieties contain no rational curves at all.
Although complex K3 surfaces contain no moving families of 
rational curves, we have the well-known

\begin{Conjecturex}
 Every projective K3 surface over an algebraically closed field 
 contains infinitely many integral rational curves.
\end{Conjecturex}

Bogomolov and Mumford \cite{mori mukai} showed that every complex projective K3 
surface contains at least one rational curve.
Next, Chen \cite{chen} established existence of infinitely many rational
curves on very general complex projective K3 surfaces.
Since then, infinitely many rational curves have been established on
polarized K3 surfaces of degree $2$ and Picard rank $\rho=1$ \cite{bht},
elliptic K3 surfaces \cite{bt density}, and K3 surfaces 
with infinite automorphism groups.
In particular, this includes all K3 surfaces with 
$\rho\geq5$, as well as ``most'' K3 surfaces with $\rho\geq3$,
see \cite{bt density}.
In this article, we prove 

\begin{Theoremx}
 A complex projective K3 surface with odd Picard rank 
 contains infinitely many integral rational curves.
\end{Theoremx}

Our proof uses the approach of Bogomolov, Tschinkel, and Hassett
from \cite{bht}, i.e., reduction to positive characteristic and  
moduli spaces of stable maps.
More precisely, many essential steps to proving our result have already
been carried out in
\cite{bht}, where the key observation was made that K3 surfaces
over $\overline{\FF}_p$ have even Picard ranks.
The difficulty faced in \cite{bht} comes from the deformation of
non-reduced curves.
We overcome this difficulty by adding so-called {\em rigidifiers} 
to these non-reduced curves, see Section \ref{sec: isolated}.
Our techniques also yield the following 
result in positive characteristic:

\begin{Theoremx}
 A non-supersingular K3 surface with odd Picard rank
 over an algebraically closed field of characteristic $p\geq5$ 
 contains infinitely many integral rational curves.
\end{Theoremx}

The article is organized as follows:

In Section \ref{sec: generalities}, we recall a couple of 
results about rational curves on K3 surfaces.
Also, we extend them to characteristic $p$, which is 
probably known to the experts.

In Section \ref{sec: isolated}, we discuss rigid
genus zero stable maps and 
introduce the notion of {\em rigidifiers}.
Rigidifiers are genus zero stable maps that have
the property that any sum of rational curves
on a K3 surface can be represented by a rigid
genus zero stable map after adding sufficiently many
rigidifiers.
In this article, {\em rigid} means that we
allow infinitesimal deformations but no one-dimensional
non-trivial families.

In Section \ref{sec: main}, we prove our main result.
By \cite{bht}, it suffices to establish it for K3 surfaces
over number fields.
For such surfaces, we find rational curves of arbitrary
high degree on reductions modulo $p$.
Next, we deform the surface and its high degree curve
to a nearby surface that contains rigdifiers.
Then, we use these rigidifiers to deform, as well as to
lift to characteristic zero.

\begin{Acknowledgements}
 We thank the referee for remarks and comments.
 The first named author is partially supported by DARPA HR0011-08-1-0091
 and NSF grant NSF-0601002.
 The second named author gratefully acknowledges funding from DFG under 
 research grant LI 1906/1-2 and thanks the department of mathematics
 at Stanford university for kind hospitality.
\end{Acknowledgements}

\section{Generalities}
\label{sec: generalities}

In this section we review general results about rational curves
on K3 surfaces.
On our way, we extend these to characteristic $p$, which
is probably known to the experts.

\begin{Theorem}[Bogomolov-Mumford $+\,\varepsilon$]
 \label{rational curve in linear system}
 Let $X$ be a projective K3 surface over an algebraically closed field $k$.
 Let ${\cal L}$ be a non-trivial, effective and invertible sheaf. 
 Then there exists a divisor in $|{\cal L}|$ that is a sum of
 rational curves.
\end{Theorem}

\begin{proof}
Since $X$ is a K3 surface, ${\cal L}$ is isomorphic to 
${\cal L}'\otimes\OO_X(\sum_i a_iC_i)$, where the $a_i$ are positive
integers, the $C_i$ are smooth rational curves, and ${\cal L}'$ is a
nef invertible sheaf.
Replacing ${\cal L}$ by ${\cal L}'$ we may assume
that ${\cal L}$ is non-trivial, nef and satisfies ${\cal L}^2\geq0$.

If ${\cal L}^2=0$ then $|{\cal L}|$ defines a genus-one fibration $X\to\PP^1$.
Since $X$ is K3, not all fibers are smooth.
In particular, there exists
a fiber, i.e., a divisor in $|{\cal L}|$,
that is a sum of rational curves.

Next, we assume ${\cal L}^2>0$, i.e., ${\cal L}$ is big and nef.
In characteristic zero, our assertion is shown 
in \cite[Proposition 2.5]{bt density}.
If ${\rm char}(k)=p>0$,
then there exists a possibly ramified extension $R$
of the Witt ring $W(k)$ such that the pair $(X,{\cal L})$ lifts to 
a formal scheme over $\Spf R$ by \cite[Corollaire 1.8]{deligne}.
We write this as a limit of schemes $X_n\to\Spec R_n$.
Then, for all $n\geq0$
$$
  X_n'\,:=\,\Proj \bigoplus_{k\geq0}H^0(X_n,{\cal L}_n^{\otimes k})
\,\longrightarrow\,\Spec R_n
$$
is a projective surface, whose special fiber
$X'=X_0'\to\Spec k$ has at worst Du~Val singularities.
Since each $\OO_{X_n'}(1)$ is ample on $X_n'$ we obtain an
ample invertible sheaf on the limit, which is algebraizable
by Grothendieck's existence theorem.
We thus obtain a scheme ${\cal X}'\to\Spec R$ lifting $X'$.
By \cite{Artin simultaneous},
there exists a possibly ramified extension $R\subseteq R'$ and
a smooth algebraic space $\widetilde{{\cal X}}\to\Spec R'$
with special fiber $X$. 
The ample invertible sheaf on ${\cal X}'$ pulls back to 
an invertible sheaf $\widetilde{{\cal L}}$ on $\widetilde{{\cal X}}$,
which lifts ${\cal L}$.
Applying \cite[Proposition 2.5]{bt density} to $\widetilde{{\cal L}}$
and reducing modulo $p$, we find a divisor in $|{\cal L}|$ that
is a sum of rational curves.
\end{proof}

As corollary, we obtain a result of 
Mori and Mukai \cite{mori mukai} that they attribute to
Bogomolov and Mumford, see also
\cite[Corollary 18]{bht}.

\begin{Theorem}[Bogomolov-Mori-Mukai-Mumford $+\varepsilon$]
 A projective K3 surface over an algebraically closed field 
 contains a rational curve. \qed
\end{Theorem}

Moreover, it is commonly believed that 

\begin{Conjecture}
 \label{main conjecture}
 Every projective K3 surface over an algebraically closed field 
 contains infinitely many integral rational curves.
\end{Conjecture}

Chen \cite{chen} has shown that this is true for
very general complex projective K3 surfaces, 
but see also the discussion in \cite[Section 3]{bht}.
Moreover, let us mention the following reduction to the
case of number fields:

\begin{Theorem}[Bogomolov-Hassett-Tschinkel {\cite[Theorem 3]{bht}}]
 Assume that for every K3 surface $X$ defined over a number field $K$,
 there are infinitely many rational curves on 
 $$
   X_{\overline{\QQ}} \,:=\, X\otimes_K\overline{\QQ}\,.
 $$
 Then, Conjecture \ref{main conjecture} holds for algebraically closed
 fields of characteristic zero. \qed
\end{Theorem}

\begin{Remark}
  \label{bht remark}
  Moreover, the proof of \cite[Theorem 3]{bht} shows that Conjecture \ref{main conjecture}
  holds for projective K3 surfaces with Picard rank $\rho_0$ over algebraically
  closed fields of characteristic zero if it holds for K3 surfaces with
  Picard rank $\rho_0$ over $\overline{\QQ}$.
\end{Remark}

\section{Rigid stable maps and rigidifiers}
\label{sec: isolated}

In this section we first give a criterion for a genus zero stable map to a projective 
K3 surface to be rigid.
Then, we introduce a class of stable maps, called {\em rigidifiers},
that has the property that given any sum of rational curves on a K3 surface, this sum
can be represented by a rigid stable map after adding sufficiently
many rigidifiers.
Finally, we show that surfaces with rigidifiers form an open and dense subset
inside the moduli space of polarized K3 surfaces.

\begin{Definition}
A morphism $f:C\to X$, where $C$ is a proper and connected curve
with at worst nodal singularities, 
is called {\em stable}, if ${\rm Aut}(f)$ is finite.
Here, ${\rm Aut}(f)$ denotes the automorphism group scheme of
all automorphisms of $C$ that commute with $f$.
\end{Definition}

For a projective K3 surface $(X,H)$, there exists a moduli space of stable maps \cite{kontsevich}.
Various algebraic constructions are discussed in \cite{abramovich vistoli},
and a formal Artin stack over the formal deformation space of K3 surfaces
is constructed in the proof of \cite[Theorem 19]{bht}. 
In this paper, without further mentioning, all domain curves of stable maps
will have arithmetic genus zero. 
For an integer $\beta$, we denote by $\overline{{\cal M}}_{0}(X,\beta)$ the moduli 
stack of genus zero 
stable maps $[f,C]$ of degree $\beta$, i.e., stable genus zero
maps $f:C\to X$ with $\deg f_\ast[C]=\beta$.

\begin{Definition}
Let $[f,C]$ be a stable map and $D$ a sum of rational curves.
\begin{enumerate}
  \item  We call $[f,C]$ {\em rigid} if $\overline{{\cal M}}_{0}(X, \beta)$
    is zero-dimensional at $[f]$.
  \item We say that the rational curve $D=\sum_i D_i$ on $X$ has a 
     {\em rigid representative} (by a stable map) if there 
     exists a rigid stable map $[f,C]\in \overline{{\cal M}}_{0}(X,\beta)$
     such that $f_\ast [C]=[D]$.
\end{enumerate}
\end{Definition}

\begin{Remark}
  Here, we use the word ``rigid'' in the most liberal manner, i.e.,
  for a rigid stable map $[f,C]$ the morphism $f$ may admit infinitesimal
  but no one-dimensional deformations in $\overline{{\cal M}}_{0}(X,\beta)$.
\end{Remark}

For example, any integral rational curve $D$ on
a non-supersingular K3 surface $X$ has a rigid representative, namely
via its normalization $\nu:\widetilde{D}\to D\subseteq X$.
On the other hand, no multiple of an irreducible smooth
rational curve has a rigid representative, as any representative
$[f]$ must involve multiple covers that deform and give rise to a positive 
dimensional component of $\overline{{\cal M}}_{0}(X,\beta)$ through $[f]$.

Let us now introduce some useful notions for genus zero stable maps.

\begin{Definition}
 Let $[f,C]$ be a stable map.
 \begin{enumerate} 
   \item An irreducible component $\Sigma\subseteq C$ is 
      a {\em ghost-component} if $f(\Sigma)$ is a point.
   \item Two irreducible components $\Sigma_1,\Sigma_2\subseteq C$ are   
       {\em adjacent} if either $\Sigma_1\cap\Sigma_2\neq\emptyset$ or if
       they are connected by a chain of ghost-components of $C$. 
   \item For two adjacent components $\Sigma_1$ and $\Sigma_2$ we call 
        $p_i\in \Sigma_i$  their {\em intersection points} 
        if $p_i\in \Sigma_i$ is the intersection $\Sigma_1\cap \Sigma_2$
        if non-empty, or if $p_i=\Sigma_i\cap B$ for some chain  
        $B\subseteq C$ of ghost-components.
 \end{enumerate} 
\end{Definition}

Let $\Sigma_1,\Sigma_2\subseteq C$ be two non-ghost adjacent components, 
and let $p_i\in \Sigma_i$ be their intersection points. 
Let $\widehat{\Sigma}_i$ be the formal completion of $\Sigma_i$ at $p_i$ and 
denote by $\widehat{f}_i: \widehat{\Sigma}_i\to X$ the induced morphism.

\begin{Definition}
\label{adj}
  Two adjacent non-ghost components $\Sigma_1$ and $\Sigma_2$ intersect
  {\em properly at their intersection in $X$} if for
  $p_i\in\Sigma_i$ and $[\widehat{f}_i,\widehat{\Sigma}_i]$ just
  mentioned, $\widehat{f}_1^{-1}(\widehat{f}_2(\widehat{\Sigma}_2))
  \subseteq\widehat{\Sigma}_1$ is non-trivial and zero-dimensional.
\end{Definition}

Next, we establish a criterion for a stable map to be rigid.
Although we will apply it later only to genus zero stable maps that 
contain no ghost-components, we prove a more general result, which may
be useful in future applications.

\begin{Lemma} 
\label{rigid}
Let $X$ be a non-supersingular K3 surface.
Let $[f,C]$ be a genus zero stable map to $X$. 
Then $[f,C]$ is rigid if the following conditions hold:
\begin{enumerate}
\item for every non-ghost-component $\Sigma\subseteq C$, $f|_{\Sigma}:\Sigma\to f(\Sigma)$ 
is birational,
\item ghost-components of $[f,C]$ are disjoint, and every ghost-component contains
exactly three nodal points of $C$, and  
\item any two adjacent non-ghost-components intersect properly at their intersection in $X$.
\end{enumerate}
\end{Lemma}

\begin{proof} 
Let $f:{\cal C}\to X$ with ${\cal C}\to S$ be an $S$-family of stable maps
over a smooth and irreducible curve $S$.
Assume that the special fiber $f_0: {\cal C}_0 \to X$ over $0\in S$ 
satisfies the conditions (1) - (3).
We will show that $f$ is a constant family of stable maps
over an open and dense neighborhood of $0\in S$.

We denote by ${\cal C}_a$, $a\in A$ the irreducible components of $\cal C$. 
For a general closed point $s\in S$ and $a\in A$ we consider the fiber 
${\cal C}_{a,s}:={\cal C}_a\times_S s$.
Since $\cC_{a,0}$ has arithmetic genus $0$, each $\cC_{a,s}$ is 
isomorphic to $\PP^1$.

We distinguish two cases: 
first, assume that $f({\cal C}_{a,s})$ is a curve, which is necessarily irreducible.
Since the image of $f$ is a rational curve, and $X$ is a non-supersingular K3 surface,
the image does not move, i.e., does not depend on $s$ for all $s\neq 0$.
We denote this image by $R_a$. We claim that $f: \cC_{a,s}\to R_a$ is birational.

Indeed, since ${\cal C}_a$ is proper and flat over $S$, we conclude $f({\cal C}_{a,0})=R_a$. 
We next show that ${\cal C}_{a,0}$ contains only one non-ghost-component:
let $\Sigma_1,\ldots,\Sigma_r$ be the non-ghost-components of ${\cal C}_{a,0}$.
We have $R_a=f({\cal C}_a)$ and denote by $\widetilde{R}_a$ its normalization. 
Then $f:\Sigma_i\to R_a$ lifts uniquely to $h_i: \Sigma_i\to \widetilde{R}_a$. 
After possibly reindexing, we may suppose that $\Sigma_1$ and $\Sigma_2$ are adjacent.
By condition (3), they intersect properly at $p_1\in\Sigma_1$ and 
$p_2\in\Sigma_2$ under $f$. 
Let $q=f(p_1)=f(p_2)\in R_a$, and let $\widehat{R}_a$ be the formal completion 
of $R_a$ at $q$. 
By the proper intersection assumption, the images 
$h_1(\widehat{\Sigma}_1)$ and $h(\widehat{\Sigma}_2)$ do not lie
on the same branch of $\widehat{R}_a$, which shows that
$h_1(p_1)$ and $h_2(p_2)$ are distinct.

On the other hand, since $S$ is smooth, ${\cal C}_a$ is normal. 
Thus, $f:{\cal C}_a\to R_a$ 
uniquely lifts to $\widetilde{f}_a: {\cal C}_a\to\widetilde{R}_a$.
Since $\widetilde{f}_a|_{\Sigma_i}$ coincides with 
$h_i|_{\Sigma_i}$ at general points of $\Sigma_i$, 
they are identical. 
Therefore, 
$h_1(p_1)=\widetilde{f}_a(p_1)=\widetilde{f}_a(p_2)=h_2(p_2)$,
which contradicts $h_1(p_1)\neq h_2(p_2)$. 
This proves that ${\cal C}_a$ contains precisely one 
non-ghost-component. 
Therefore, by assumption (1), $f: \cC_{a,0}\to R_a$ is generically bijective. 
Since $f$ is flat over $S$, $f: \cC_{a,s}\to R_a$ is also generically bijective, 
and we conclude birationality.

The second case is when $f|_{\cC_{a,s}}$ is a constant map. 
Since $f$ is flat, $f|_{\cC_{a,0}}$ is a constant map, too. 
By assumption (2), $\cC_{a,0}$ is irreducible, and therefore contains exactly three 
nodal points of $\cC_0$. 
This proves that for general $s\in S$, $\cC_{a,s}$ contains three nodes of $\cC_a$.

We now show that for a Zariski open and dense subset $U\subseteq S$, 
$f|_{\cC_U}$ is a constant family of stable maps, where 
$\cC_U:=\cC\times_S U$.
By the previous discussion, we find a
dense open subset $U\subseteq S$ so that 
$\cC_{a,U}=\cC_a\times_S U\cong \PP^1\times U$. 
We next study the nodal points of $\cC_s$. Let $T_{ab}=\cC_{a}\cap \cC_{b}$. 
Since $\cC$ is a family of arithmetic genus zero curves,
$T_{ab}$ is either a section of $\cC\to S$ or empty. 
Let $\pi: \cC\to S$ be the projection.

Suppose $T_{ab}\ne\emptyset$ and that both, $\cC_a$ and $\cC_b$, are not families 
of ghost-components.
We set $\cC_{ab}:=\cC_a\cup\cC_b\subseteq \cC$, and conclude that
$(f,\pi): \cC_{ab}\to (R_a\cup R_b)\times S$ is generically finite.
We denote by $\cC_{ab}^{\text{st}}$ the contraction of the exceptional divisor of
$(f,\pi): \cC_{ab}\to X\times S$. 
Then $\cC_{ab}^{\text{st}}=(\PP^1\sqcup \PP^1)\times S$,
where $\PP^1\sqcup\PP^1$ denotes the union of two $\PP^1$'s intersecting at one point.
Applying assumption (3), we see that the two irreducible components of $\cC_{ab}^{\text{st}}\times_S 0$
intersect properly at their intersection in $X$. 
By the same argument as in proving that each $\cC_a$ contains at most
one non-ghost component, 
we conclude that $f(T_{ab})\subseteq\text{Sing}(R)$, where $R=\cup_a R_a$ with the
reduced structure. 
Using $\cC_{a,U}\cong \widetilde{R}_a\times U$ from above, we see that
$$
T_{ab}\times_S U\,\subseteq\, \cC_{a,U}\,\cong\, \widetilde{R}_a\times U
$$ 
is a constant section. 

In case $T_{ab}\ne\emptyset$ and $\cC_b$ is a family of ghost components, 
(by assumption (2), $\cC_a$ cannot be a family of ghost-components,)
there must be a third component $\cC_c$ so that $T_{bc}\ne \emptyset$. 
We set $\cC_{ac}:=(\cC_a\cup\cC_c)/\sim$, where $\sim$ means that we 
identify $T_{ab}\subset \cC_a$ with $T_{bc}\subset \cC_c$. 
The morphism $f$ restricted to $\cC_a$ and $\cC_c$ 
defines an $S$-family of stable maps $f_{ac}: \cC_{ac}\to X$.
By the same arguments as before, we conclude that 
$T_{ab,U}\subset \cC_{a,U}\cong \widetilde{R}_a\times U$
is a constant family.

We conclude that the restricted family $f: \cC|_U\to X$ is a constant family 
of stable maps over $U$. 
Finally, for all $s\in U\cup\{0\}$, the restriction of $f_s$ to a 
non-ghost-component is birational onto its image, and so 
$f_s$ is tame.
Since the moduli space of tame stable maps is separated, 
the family $f:\cC\to X$ is constant over an open and dense neighborhood 
of $0\in S$.
\end{proof}

We now come to the main definition of this section, which 
will be motivated by Theorem \ref{rigidifiers produce rigid maps} below.

\begin{Definition}
 \label{def:rigidifier}
 A {\em rigidifier} is a morphism $f:\PP^1\to X$ to a surface, where
 \begin{enumerate}
  \item $f:\PP^1\to D:=f_\ast\PP^1$ is the normalization morphism,
  \item $D$ is an integral rational curve with only simple nodes as singularities, and
  \item the class $f_\ast[\PP^1]$ is ample.
 \end{enumerate}
 In particular, $[f,\PP^1]$ is a genus zero stable map.
\end{Definition}

We recall that
the moduli space ${\cal M}_{2d}$ of polarized K3 surfaces of degree $2d$
exists as separated Deligne-Mumford stack of finite type over $\Spec\ZZ$,
which is even smooth over $\Spec\ZZ[\frac{1}{2d}]$, see \cite{rizov}.
For every integer $d\geq1$, we define 
\[
  U_{2d}\,:=\,\left\{ 
  (X,H)\in{\cal M}_{2d}\,\Big|\,
  \begin{array}{l}
    \mbox{ there exists an $n\in\NN$ such that $|nH|$ contains } \\
    \mbox{ a curve that can be represented by a rigidifier}
  \end{array}
  \right\}\,.
\]
Before coming to the main result of this section, 
we establish existence and openness of these stable maps.

\begin{Proposition}
 \label{prop:rigidifier open}
 For every $d\geq1$, the set $U_{2d}$ is Zariski-open and dense in 
 ${\cal M}_{2d}$. 
 Moreover, $U_{2d}$ is of finite type over $\Spec\ZZ$.
\end{Proposition}

\begin{proof}
Non-emptiness of $U_{2d}$ for all $d\geq1$ follows from \cite[Theorem 1.2]{chen}.

Given a family $({\cal X},{\cal H})\to S$ of polarized K3 surfaces and a
rigid stable map of genus zero $f_b:C\to{\cal X}_b$ with 
$f_{b \ast}C\in|nH_b|$ for some $n\in\NN$, it follows from
the proof of \cite[Theorem 19]{bht} that the component of
$\overline{{\cal M}}_0({\cal X}/S, n{\cal H})$ that contains
$[f_b]$ is proper and surjective over $S$.
(In the analytic setting this follows from \cite{ran semiregularity}.)

By definition, the image $D\subset X$ of a rigidifier is an irreducible curve
with only simple nodes as singularities.
Clearly, if $X$ is not unirational, then there is no 
one-dimensional and non-trivial family of rational curves containing $D$.
But even if $X$ is unirational, which may happen in positive 
characteristic, the generic member of such a family must
have unibranch singularities by \cite{tate}, and thus cannot
contain $D$ as special member.
We conclude that rigidifiers are rigid stable maps and thus extend over ${\cal M}_{2d}$.

The image of a rigidifier stable map is an integral
curve, which is an open property.
Next, having only simple nodes as singularities
is an open property: the arithmetic genus of the
image curve is constant in flat families, and so 
singularities cannot smoothen out.
This shows that $U$ is open.

Being open in a Noetherian Deligne-Mumford stack
of finite type over $\ZZ$, also $U$ is of finite type over $\ZZ$.
\end{proof}


\begin{Theorem}
 \label{rigidifiers produce rigid maps} 
 Let $X$ be a non-supersingular K3 surface over an algebraically closed field.
 Let $D_1,...,D_m$ be integral rational curves on $X$, not necessarily distinct,
 and let $[f,\PP^1]$ be a rigidifier.
 Then, for some $k\leq m$
 \[
   D_1\,+\,...\,+\,D_m\,+\, k\cdot f_\ast \PP^1 
 \]
 has a rigid representative.
\end{Theorem}

\begin{proof}
Let us explicitly construct the rigid representative:

First, we choose $q_1,q_2\in\PP^1$ such that $q_1,q_2$ are distinct points
mapping to the same node $f(q_1)=f(q_2)$.
Then, we take $m$ copies $[f_i,C_i]$, $i=1,...,m$ of $[f,\PP^1]$ and construct
a new stable map $[\widetilde{f},\widetilde{C}]$
by connecting $q_1\in C_i$ to $q_2\in C_{i+1}$ for all $i=1,...,m-1$.
We note that $q_1\in C_i$ and $q_2\in C_{i+1}$ intersect properly in $X$.

For all $i$, we represent $D_i$ via the stable map $[\nu_i,\PP^1]$
coming from the normalization $\nu_i:\PP^1\to D_i$.
Next, we insert $[\nu_i,\PP^1]$ into $[\widetilde{f},\widetilde{C}]$
by attaching it to $[f_i,C_i]$, as follows:
since $f_\ast\PP^1$ is ample, it intersects every $D_i$. 
Then, at least one of the following cases is fulfilled, which
gives us a recipe to build $\nu_i:\PP^1\to D_i$ into the rigid stable
map we want to construct.

(1) Assume that $D_i=\nu_{i,\ast}\PP^1$ is not equal to $f_\ast\PP^1$ and that
$\nu_i(\PP^1)$ intersects $\nu_i(C_i)$ in $\nu_i(p)=f_i(q)$, say.
If $q\in C_i$ is a smooth point of $\widetilde{C}$, then
we simply add $[\nu_i,\PP^1]$ to $[f_i,C_i]$ by connecting $q$ to $p$.

(2) Assume that $D_i=\nu_{i,\ast}\PP^1$ is not equal to $f_\ast\PP^1$, and that
$\nu_i(\PP^1)$ intersects $\nu_i(C_i)$ in $\nu_i(p)=f_i(q)$. 
But now, assume that $q\in C_i$ is a node of $\widetilde{C}$.
\begin{enumerate}
 \item[(a)]  If $q=C_{i-1}\cap C_i$ then we let $q'$ be the other point of $C_i$
   mapping to the node $f_i(q)$.
   If $q'$ is not a node of $\widetilde{C}$, we add $[\nu_i,\PP^1]$ to
   $[f_i,C_i]$ by connecting $q'$ to $p$.
   However, if $q'$ is a node of $\widetilde{C}$ then it connects
   $C_i$ with $C_{i+1}$.
   In this case, we disconnect $C_i$ and $C_{i+1}$, and connect
   $[\nu_i,\PP^1]$ to $[f_i,C_i]$ by connecting $q'$ to $p$.
   Since $f_\ast\PP^1$ is a nodal rational curve of arithmetic genus
   $p_a=1+H^2/2\geq2$, it has at least two nodes.
   We use this second node to connect $C_i$ again to $C_{i+1}$.
  \item[(b)] If $q=C_i\cap C_{i+1}$ then we disconnect $C_i$ and $C_{i+1}$,
      connect $[\nu_i,\PP^1]$ to $[f_i,C_i]$ by connecting $q$ to $p$,
      and use another node of $f_\ast\PP^1$ to connect $C_i$  
      to $C_{i+1}$.
\end{enumerate}

(3) Finally, assume that $D_i=\nu_{i,\ast}\PP^1$ is equal to $f_\ast\PP^1$.
Then, we simply leave out $[\nu_i,\PP^1]$ as it is already included
in $[\widetilde{f},\widetilde{C}]$ via $[f_i,C_i]$.

Inspecting the previous construction, we see that the conditions
of Lemma \ref{rigid} are fulfilled.
In particular, $[\widetilde{f},\widetilde{C}]$ is a rigid stable map.
Moreover, by construction $\widetilde{f}_\ast\widetilde{C}$ equals
$\sum_{i=1}^m D_i+k\cdot f_\ast\PP^1$ for some $k\leq m$.
\end{proof}

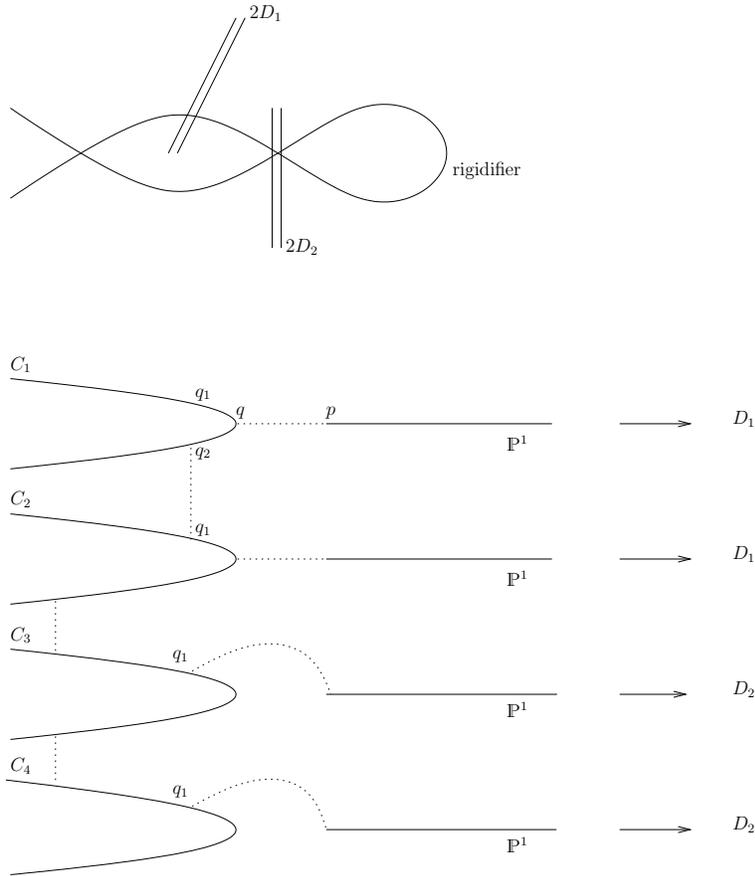
\begin{figure}[htbp]
\begin{center}
 \scalebox{0.6}{\input{rigid.pstex_t}}
\caption{To rigidify $2D_1+2D_2$ as shown above, we glue the curves below 
  along the dotted arrows to construct the desired rigid stable representative.
  The four $\PP^1$'s are mapped to $D_1$ and $D_2$ as indicated; all $C_i$
  are mapped to the rigidifier.}
\label{figure:example}
\end{center}
\end{figure}

\section{Infinitely Many Rational Curves}
\label{sec: main}

In this section we prove that complex projective K3 surfaces with 
odd Picard rank contain infinitely many rational curves, which is our main result.
We also establish it for non-supersingular K3 surfaces in characteristic $p\geq5$.

By \cite{artin mazur}, the formal Brauer
group of a K3 surface is a smooth and $1$-dimensional formal group.
In positive characteristic, its height $h$ satisfies $1\leq h\leq10$
or $h=\infty$.
Surfaces with $h=1$ are called {\em ordinary} and this property is open
in families of equal characteristic, 
whereas surfaces with $h=\infty$ are called {\em supersingular}.
If a K3 surface contains a moving family of rational curves,
then it is uniruled, and in particular, supersingular.

\begin{Theorem}[Bogomolov-Zarhin, Joshi-Rajan, Nygaard-Ogus]
 \label{thm bzno}
 Let $X$ be a K3 surface over a number field $K$. 
 Then,
 \begin{enumerate}
  \item for all but finitely many places $\idealp$ of $K$, the reduction $X_\idealp$
     is smooth,
  \item there exists a finite extension $L/K$ and a set $S$ of places of $L$
    of density $1$, such that the reduction $(X_L)_\idealq$
    is ordinary for all $\idealq\in S$,
  \item for all places $\idealp$ of characteristic $p\geq5$, where $X$ has good 
    and non-supersingular reduction, the geometric Picard rank 
    $\rho((X_\idealp)_{\overline{\FF}_p})$ is even.
 \end{enumerate}
\end{Theorem}

\begin{proof}
The first assertion follows from openness of smoothness.
The second statement is shown in \cite{bogomolov zarhin}, as well
as by Joshi and Rajan (unpublished).
The final assertion follows from the Weil conjectures and the 
results on the Tate conjecture in \cite{nygaard ogus}, see also 
\cite[Theorem 15]{bht}.
\end{proof}

The following result shows that we can find rational curves of arbitrary high
degree when reducing a surface with odd Picard rank 
modulo $p$:

\begin{Proposition}
  \label{produce high degree curve}
  Let $(X,H)$ be a polarized K3 surface over a number field $K$,
  such that $\Pic(X)=\Pic(X_{\overline{\QQ}})$ and such that the
  Picard rank is odd.
  Then, there is a finite extension $L/K$ such that
  for every $N\geq0$ there exists a set $S_N$ of places of $L$
  of density $1$ such that for all $\idealq\in S_N$
  \begin{enumerate}
   \item the reduction $(X_L)_\idealq$ is a smooth and non-supersingular K3 surface,
   \item the reduction $H_\idealq$ is ample,
   \item there exists an integral rational curve
      $D_\idealq$ on $((X_L)_\idealq)_{\overline{\FF}_p}$ such that
      \begin{enumerate}
         \item the class of $D_\idealq$ does not lie in 
           $\Pic(X)\otimes_\ZZ\QQ$, where we view
           $\Pic(X)$ as subgroup of $\Pic(((X_L)_\idealq)_{\overline{\FF}_p})$ 
           via the specialization homomorphism, and
         \item $D_\idealq\cdot H_\idealq\geq N$.
      \end{enumerate}
  \end{enumerate}
\end{Proposition}

\begin{proof}
By Theorem \ref{thm bzno}, there exists a finite extension $L/K$ and set $S$ of
primes of density $1$, such that the reduction $(X_L)_\idealq$ for all $\idealq\in S$
is smooth and not supersingular.
Since ampleness is an open property, we may assume - after possibly removing a finite
number of places from $S$ - that the reduction $H_\idealq$ is ample.

By Theorem \ref{thm bzno}, we may choose for every $\idealq\in S_N$ an invertible
sheaf ${\cal L}_\idealq$
in $\Pic(((X_L)_\idealq)_{\overline{\FF}_p})$ that
does not lie in $\Pic(X_L)\otimes_\ZZ\QQ$.
Here, we view $\Pic(X)=\Pic(X_L)=\Pic(X_{\overline{\QQ}})$ 
as subgroup of $\Pic(((X_L)_\idealq)_{\overline{\FF}_p})$
via specialization.
Without loss of generality, these ${\cal L}_\idealq$ are effective, and thus,
by Theorem \ref{rational curve in linear system},
every $|{\cal L}_\idealq|$ contains a sum $D_\idealq$ of rational curves.
Passing to an appropriate subdivisor, we may assume
that $D_\idealq$ is a geometrically integral rational curve, whose
class does not lie in $\Pic(X)\otimes_\ZZ\QQ$.

Seeking a contradiction, we assume that $D_\idealq\cdot H_\idealq<N$ for
infinitely many $\idealq\in S$.
Let ${\cal X}\to\Spec\OO_{L,T}$ be a smooth model of $X_L$, where
$\OO_{L,T}$ is the ring of integers of $L$ localized
at some set of places $T$.
The scheme 
$$
   {\rm Mor}_{<N}(\PP^1, {\cal X}),
$$
which parametrizes morphisms $f:\PP^1\to{\cal X}$ with $f_*(\PP^1)\cdot H<N$,
is of finite type over $\ZZ$.
It has $\overline{\FF}_p$-rational points
for infinitely many $p$, corresponding to all the $D_\idealq$
with $D_\idealq\cdot H_\idealq<N$.
Thus, it has a $\overline{\QQ}$-rational point, and we may
even assume that the corresponding rational curve 
$\widetilde{D}$ on $X_{\overline{\QQ}}$ 
specializes to a $D_{\idealq'}$ 
with $D_{\idealq'}\cdot H_{\idealq'}<N$.
Now, $\widetilde{D}$ gives a class in $\Pic(X_{\overline{\QQ}})=\Pic(X)$,
whereas, by construction, none of the $D_\idealq$ gives a class that
lies inside the image of the specialization homomorphism
$\Pic(X)\to\Pic(((X_L)_\idealq)_{\overline{\FF}_p})$, 
a contradiction.
Thus, after removing finitely many places from $S$ we arrive
at a set $S_N$ such that $D_\idealq\cdot H_\idealq\geq N$
for all $\idealq\in S_N$.
\end{proof}

After this preparation, we now come to our main result:

\begin{Theorem}
  \label{main}
  A projective K3 surface $X$ with odd Picard rank
  over an algebraically closed field 
  of characteristic zero contains infinitely 
  many rational curves.
\end{Theorem}

\begin{proof}
By \cite[Theorem 3]{bht} and Remark \ref{bht remark}, we may and will assume 
that $X$ is defined over $\overline{\QQ}$.
We choose a number field $K$ such that $X$ and
every class of $\Pic(X)$ is already defined over $K$.
Then, we replace $X$ by this model over $K$ and choose 
a polarization $H$, which has some degree $2d:=H^2$, say.
Let ${\cal M}_{2d}$ be the corresponding
moduli space of polarized K3 surfaces, which exists as a separated
Deligne-Mumford stack over $\ZZ$ by \cite{rizov}.

We choose an arbitrary positive integer $N$.
Our theorem follows, if we find
an integral rational curve $D\subset X_{\overline{\QQ}}$ 
with $D\cdot H\geq N$.

Let $L/K$ and $S_N$ be as in Proposition \ref{produce high degree curve}
and replace $X$ by $X_L$.
After possibly removing a finite set of places from $S_N$, 
we find a model ${\cal X}\to \Spec \OO_{L,S_N}$ of $X$ over the ring 
of integers of $L$ localized at $S_N$.
After passing to a Zariski-open subset of ${\cal M}_{2d}$ 
and taking an appropriate finite and \'etale
cover, we arrive at a polarized family of K3 surfaces
$({\cal Y},{\cal H})\to U$
containing ${\cal X}\to\Spec\OO_{L,S_N}$, and
where $U$ is a scheme of finite type over $\ZZ$.
We note that there exists a proper algebraic stack
of stable maps $\overline{{\cal M}}_0({\cal Y}/U)\to U$, see
\cite{abramovich vistoli}.
(The reason for passing to the scheme $U$ rather than working
with ${\cal M}_{2d}$ is to avoid certain technicalities when working
with moduli of stable maps over a base that is a
Deligne-Mumford stack, see \cite[Section (1.3)]{abramovich vistoli}.)

Let $f_\idealq:\PP^1\to{\cal X}_\idealq$ be the stable map
representing $D_\idealq$, which is rigid, as 
$X_\idealq$ is not supersingular.
Thus, the relative dimension of 
$\overline{{\cal M}}_0({\cal Y}/U)\to U$
at $[f_\idealq]$ is zero.
We choose a component of $\overline{{\cal M}}_0({\cal Y}/U)$
through $[f_\idealq]$ and denote its image in $U$ by
$B_\idealq$.
Deformation theory of morphisms implies that
$\dim B_\idealq\geq\dim U-1$.
However, the invertible sheaf $\OO_{X_\idealq}(D_\idealq)$ cannot extend
over all of $U$, which implies that $B_\idealq$ is
a divisor in $U$.
To keep things simpler, we replace $B_\idealq$ by an
irreducible divisor through $X_\idealq$.
If $B_\idealq$ were not flat over $\ZZ$, it would
be contained completely in characteristic $p$,
in which case $\OO_{X_\idealq}(D_\idealq)$ would
extend to ${\cal M}_{2d}\otimes_\ZZ\FF_p$.
However, the formal divisor inside the formal deformation 
space $\Spf W(k)[[x_1,...,x_{20}]]$ of $X_\idealq$ along
which $\OO_{X_\idealq}(D_\idealq)$ extends, is flat
over $W(k)$ by \cite[Corollaire 1.8]{deligne}.
This implies that
$\OO_{X_\idealq}(D_\idealq)$ cannot extend over 
${\cal M}_{2d}\otimes_\ZZ\FF_p$.
Thus, $B_\idealq$ is a divisor in $U$ that is flat over $\ZZ$.

We claim that for every $\idealq\in S_N$, there are only
finitely many $\idealq'\in S_N$ such that $B_\idealq=B_{\idealq'}$:
if not, ${\cal X}\to\Spec\OO_{L,S_N}$ would specialize
into $B_\idealq$ for infinitely many places $\idealq\in S_N$.
This would imply that also the generic fiber $X$ is a point of
$B_\idealq$, i.e., $D_\idealq$ and the invertible sheaf
$\OO_{X_\idealq}(D_\idealq)$ would lift from $X_\idealq$
to $X$, a contradiction.
Thus, the $B_\idealq$'s form a set with infinitely many
distinct divisors in $U$.
By the openness result Proposition \ref{prop:rigidifier open},
almost all of these $B_\idealq$'s contain surfaces with
rigidifiers in some multiple of their polarization.
Let $\idealq\in S_N$ be such a place.

The curve $D_\idealq\subset X_\idealq$ extends to a rational
curve along $B_\idealq$ and on an open dense subset
this extension will be an integral curve.
Let $Z$ be a non-supersingular K3 surface on $B_\idealq$ 
such that $D_\idealq$ extends to some integral rational curve 
$D$ on $Z$ and such that $Z$ contains a 
rigidifier $\PP^1\to R\subset Z$ with $R\in|rH|$
for some $r\in\NN$.
Next, for a sufficiently large integer $m$, the linear
system $|mH-D|$ is effective.
By Theorem \ref{rational curve in linear system}, there exist
integral rational curves $D_i$, such that $\sum_i D_i$
lies in $|mH-D|$.
By Theorem \ref{rigidifiers produce rigid maps}, 
there exists an integer $k$ and a {\em rigid} stable map 
$[f_Z]\in\overline{{\cal M}}_0(Z, (m+kr)H)$
representing $D+\sum_i D_i + kR$.

Now, $\overline{{\cal M}}_0({\cal Y}/U, (m+kr){\cal H})$ is at least
$19$-dimensional, as shown in the proof of \cite[Theorem 19]{bht}.
It is proper over $U$, which is also $19$-dimensional.
Since the fiber above $Z\in U$ at $[f_Z]$
is zero-dimensional, we can extend $[f_Z]$ over the whole of $U$.

There exists a component ${\cal M}$ of 
$\overline{\cal M}_0({\cal Y}/U, (m+kr){\cal H})\to U$
containing $[f_Z]$.
Also, there exists a connected family $[f_t]$  
of stable maps in ${\cal M}$, containing $[f_Z]$ and
whose limit $f_{X_\idealq}:C_\idealq\to X_\idealq$ over 
$X_\idealq\in U$ 
contains $D_\idealq$ in its image.
We pass to the Stein factorization of 
$\overline{\cal M}_0({\cal Y}/U, (m+kr){\cal H})\to U$, and let
${\cal M}'$ be the component that contains 
this family $[f_t]$.
Then, we choose some point $[f_X]\in{\cal M}'$ lying above
the surface $X$.
The corresponding stable map $f_X:C_X\to X$ specializes to
some stable map on $X_\idealq$.
Now, since ${\cal M}'$ has connected fibers,
the specialization of $[f_X]$ modulo $\idealq$ is a deformation
of $[f_{X_\idealq}]$.
But $X_\idealq$ is not supersingular, and so the two stable
maps have the same image curve.
In particular, the image $f_X(C_X)$ contains an integral rational curve
$\widetilde{D}$ that contains $D_\idealq$ in its specialization.
We compute
$\widetilde{D}\cdot H\geq D_\idealq\cdot H_\idealq\geq N$, which establishes
existence of an integral rational curve of degree $\geq N$. 
\end{proof}

We finish with a characteristic $p$ version
of Theorem \ref{main} for polarized K3 surfaces.

\begin{Theorem}
  \label{main p>0}
  A non-supersingular 
  K3 surface with odd Picard rank 
  over an algebraically closed field of characteristic $p\geq5$
  contains infinitely many rational curves.
\end{Theorem}

\begin{proof}
As explained in \cite[Theorem 15]{bht}, $X$ cannot be defined
over a finite field. 
Thus, we may assume that $X$
is defined over the function field of a variety $B$
with $\dim B\geq1$
over some finite field $\FF_q\supseteq\FF_p$.
After possibly shrinking $B$,
we find a smooth fibration of K3 surfaces ${\cal X}\to B$
with generic fiber $X$. 
By closedness of supersingularity \cite{artin supersingular},
we may, after possibly shrinking $B$ further, assume
that no fiber is supersingular.

By \cite[Theorem 15]{bht}, the 
Picard rank of $({\cal X}_\idealq)_{\overline{\FF}_p}$ for every closed 
point $\idealq$ of $B$ is even.
Thus, given $N$, a straight forward adaption of 
Proposition \ref{produce high degree curve} shows that
for almost all closed points $\idealq\in B$ there exists 
an integral rational curve
$D_\idealq\subset ({\cal X}_\idealq)_{\overline{\FF}_p}$ 
with $D_\idealq\cdot H\geq N$.

>From here we argue as in the proof of Theorem \ref{main}
to produce an integral rational curve $D$ with $D\cdot H\geq N$
on $X$.
We leave the details to the reader.
\end{proof}

\begin{Remark}
Let us comment on the assumptions:
\begin{enumerate}
 \item According to conjectures of Artin, Mazur and Tate,
   supersingular K3 surfaces should satisfy $\rho=b_2=22$,
   see, for example, \cite{artin supersingular}.
   In particular, there should exist no supersingular K3 
   surfaces with odd Picard rank.
 \item Using \cite{nygaard} in the proof of Proposition \ref{produce high degree curve},
   and openness of ordinarity in equal characteristic,
   we see that Theorem \ref{main p>0} also holds for ordinary K3 surfaces
   with odd Picard rank in characteristic $p=2,3$.
\end{enumerate}
\end{Remark}

\end{document}

%% file: rigid.pstex_t
\begin{picture}(0,0)%
\includegraphics{rigid.pstex}%
\end{picture}%
\setlength{\unitlength}{4144sp}%
\begingroup\makeatletter\ifx\SetFigFont\undefined%
\gdef\SetFigFont#1#2#3#4#5{%
  \reset@font\fontsize{#1}{#2pt}%
  \fontfamily{#3}\fontseries{#4}\fontshape{#5}%
  \selectfont}%
\fi\endgroup%
\begin{picture}(7764,8682)(844,-8173)
\put(3286,389){\makebox(0,0)[lb]{\smash{{\SetFigFont{12}{14.4}{\rmdefault}{\mddefault}{\updefault}{\color[rgb]{0,0,0}$2D_1$}%
}}}}
\put(3646,-1951){\makebox(0,0)[lb]{\smash{{\SetFigFont{12}{14.4}{\rmdefault}{\mddefault}{\updefault}{\color[rgb]{0,0,0}$2D_2$}%
}}}}
\put(5851,-3931){\makebox(0,0)[lb]{\smash{{\SetFigFont{12}{14.4}{\rmdefault}{\mddefault}{\updefault}{\color[rgb]{0,0,0}$\mathds{P}^1$}%
}}}}
\put(5851,-5281){\makebox(0,0)[lb]{\smash{{\SetFigFont{12}{14.4}{\rmdefault}{\mddefault}{\updefault}{\color[rgb]{0,0,0}$\mathds{P}^1$}%
}}}}
\put(5851,-6586){\makebox(0,0)[lb]{\smash{{\SetFigFont{12}{14.4}{\rmdefault}{\mddefault}{\updefault}{\color[rgb]{0,0,0}$\mathds{P}^1$}%
}}}}
\put(5851,-7936){\makebox(0,0)[lb]{\smash{{\SetFigFont{12}{14.4}{\rmdefault}{\mddefault}{\updefault}{\color[rgb]{0,0,0}$\mathds{P}^1$}%
}}}}
\put(5311,-1186){\makebox(0,0)[lb]{\smash{{\SetFigFont{12}{14.4}{\rmdefault}{\mddefault}{\updefault}{\color[rgb]{0,0,0}$\mbox{rigidifier}$}%
}}}}
\put(901,-3121){\makebox(0,0)[lb]{\smash{{\SetFigFont{12}{14.4}{\rmdefault}{\mddefault}{\updefault}{\color[rgb]{0,0,0}$C_1$}%
}}}}
\put(901,-4471){\makebox(0,0)[lb]{\smash{{\SetFigFont{12}{14.4}{\rmdefault}{\mddefault}{\updefault}{\color[rgb]{0,0,0}$C_2$}%
}}}}
\put(901,-5821){\makebox(0,0)[lb]{\smash{{\SetFigFont{12}{14.4}{\rmdefault}{\mddefault}{\updefault}{\color[rgb]{0,0,0}$C_3$}%
}}}}
\put(2746,-3976){\makebox(0,0)[lb]{\smash{{\SetFigFont{12}{14.4}{\rmdefault}{\mddefault}{\updefault}{\color[rgb]{0,0,0}$q_2$}%
}}}}
\put(2746,-4741){\makebox(0,0)[lb]{\smash{{\SetFigFont{12}{14.4}{\rmdefault}{\mddefault}{\updefault}{\color[rgb]{0,0,0}$q_1$}%
}}}}
\put(2746,-3391){\makebox(0,0)[lb]{\smash{{\SetFigFont{12}{14.4}{\rmdefault}{\mddefault}{\updefault}{\color[rgb]{0,0,0}$q_1$}%
}}}}
\put(2521,-6001){\makebox(0,0)[lb]{\smash{{\SetFigFont{12}{14.4}{\rmdefault}{\mddefault}{\updefault}{\color[rgb]{0,0,0}$q_1$}%
}}}}
\put(2521,-7351){\makebox(0,0)[lb]{\smash{{\SetFigFont{12}{14.4}{\rmdefault}{\mddefault}{\updefault}{\color[rgb]{0,0,0}$q_1$}%
}}}}
\put(4051,-3571){\makebox(0,0)[lb]{\smash{{\SetFigFont{12}{14.4}{\rmdefault}{\mddefault}{\updefault}{\color[rgb]{0,0,0}$p$}%
}}}}
\put(3151,-3571){\makebox(0,0)[lb]{\smash{{\SetFigFont{12}{14.4}{\rmdefault}{\mddefault}{\updefault}{\color[rgb]{0,0,0}$q$}%
}}}}
\put(901,-7126){\makebox(0,0)[lb]{\smash{{\SetFigFont{12}{14.4}{\rmdefault}{\mddefault}{\updefault}{\color[rgb]{0,0,0}$C_4$}%
}}}}
\put(8101,-3661){\makebox(0,0)[lb]{\smash{{\SetFigFont{12}{14.4}{\rmdefault}{\mddefault}{\updefault}{\color[rgb]{0,0,0}$D_1$}%
}}}}
\put(8101,-5011){\makebox(0,0)[lb]{\smash{{\SetFigFont{12}{14.4}{\rmdefault}{\mddefault}{\updefault}{\color[rgb]{0,0,0}$D_1$}%
}}}}
\put(8101,-6361){\makebox(0,0)[lb]{\smash{{\SetFigFont{12}{14.4}{\rmdefault}{\mddefault}{\updefault}{\color[rgb]{0,0,0}$D_2$}%
}}}}
\put(8101,-7711){\makebox(0,0)[lb]{\smash{{\SetFigFont{12}{14.4}{\rmdefault}{\mddefault}{\updefault}{\color[rgb]{0,0,0}$D_2$}%
}}}}
\end{picture}%

%% file: rationalcurves_final.bbl
\begin{thebibliography}{XXXX}
 \bibitem[AV]{abramovich vistoli} D.~Abramovich, A.~Vistoli, {\em
    Compactifying the space of stable maps}, J. Amer. Math. Soc. 15, 27-75 (2002).
 \bibitem[Ar1]{artin supersingular} M.~Artin, {\em Supersingular K3 Surfaces},
   Ann. Sci. \'Ecole Norm. Sup. 7, 543-568 (1974).
 \bibitem[Ar2]{Artin simultaneous} M.~Artin, {\em Algebraic construction of Brieskorn's 
   resolutions}, J. Algebra 29, 330-348 (1974). 
 \bibitem[AM]{artin mazur} M.~Artin, B.~Mazur, {\em Formal groups arising from algebraic 
   varieties},  Ann. Sci. \'Ecole Norm. Sup. 10, 87-131 (1977).
 \bibitem[BT]{bt density} F.~Bogomolov, Y.~Tschinkel, {\em Density of rational points
   on elliptic K3 surfaces}, Asian J. Math.  4, 351-368  (2000).
 \bibitem[BHT]{bht} F.~Bogomolov, B.~Hassett, Y.~Tschinkel, {\em Constructing rational 
   curves on K3 surfaces},  arXiv:0907.3527 (2009).
 \bibitem[BZ]{bogomolov zarhin} F.~Bogomolov, Y.~G.~Zarhin, {\em Ordinary reduction 
   of K3 surfaces},  Cent. Eur. J. Math.  7, 206-213 (2009).
 \bibitem[Ch]{chen} X.~Chen, {\em Rational curves on K3 surfaces},  
   J. Algebraic Geom.  8, 245-278 (1999).
 \bibitem[Del]{deligne} P.~Deligne, {\em Rel\`evement des surfaces $K3$ en caract\'eristique nulle},
  Lecture Notes in Math. 868,  58-79, Springer (1981).
 \bibitem[Kon]{kontsevich} M.~Kontsevich, {\em Enumeration of rational curves via 
    torus actions}, The moduli space of curves,
    Progr. Math. 129, Birkh\"auser, 335-368, 1995. 
 \bibitem[La]{lang} S.~Lang, {\em Survey of Diophantine Geometry}, Springer (1997).
 \bibitem[MM]{mori mukai} S.~Mori, S.~Mukai, {\em The uniruledness of the moduli space
   of curves of genus $11$}, Lecture Notes in Math. 1016, 334-353, Springer (1983).
 \bibitem[Ny]{nygaard} N.~Nygaard, {\em The Tate conjecture for ordinary K3 surfaces 
   over finite fields}, Invent. Math. 74, 213-237 (1983).
 \bibitem[NO]{nygaard ogus} N.~Nygaard, A.~Ogus, {\em Tate's conjecture for 
   K3 surfaces of finite height},  Ann. of Math. 122, 461-507 (1985).
  \bibitem[Ra]{ran semiregularity} Z.~Ran, {\em Semiregularity, obstructions and 
   deformations of Hodge classes},
   Ann. Scuola Norm. Sup. Pisa Cl. Sci. 28, 809-820 (1999).
  \bibitem[Riz]{rizov} J.~Rizov, {\em Moduli stacks of polarized K3 surfaces in 
     mixed characteristic}, Serdica Math. J. 32, 131-178 (2006).
  \bibitem[Ta]{tate} J.~Tate, {\em Genus change in inseparable extensions of function fields},
    Proc. Amer. Math. Soc. 3, 400-406 (1952).
\end{thebibliography}
